\documentclass[12pt]{amsart}
\usepackage{latexsym,amsthm,amscd,euscript,float,subfig,fixmath}
\usepackage{amssymb}
\usepackage{amsmath}
\usepackage{graphicx}

\newtheorem{theorem}{Theorem}[section]
\newtheorem{lemma}[theorem]{Lemma}
\newtheorem{corollary}[theorem]{Corollary}

\theoremstyle{definition}

\newtheorem{remark}[theorem]{Remark}
\newtheorem{example}[theorem]{Example}

\numberwithin{equation}{section}

\newcounter{minutes}
\divide\time by 60
\newcounter{hours}
\multiply\time by 60
\addtocounter{minutes}{-\time}

\begin{document}
\vspace*{-2cm}
\title[Meromorphic functions with small Schwarzian derivative]
{Meromorphic functions with small Schwarzian derivative}
\def\thefootnote{}
\footnotetext{ \texttt{\tiny File:~\jobname .tex,
          printed: \number\day-\number\month-\number\year,
          \thehours.\ifnum\theminutes<10{0}\fi\theminutes}
} \makeatletter\def\thefootnote{\@arabic\c@footnote}\makeatother

\author[V. Arora]{Vibhuti Arora}
\address{Vibhuti Arora, Discipline of Mathematics,
Indian Institute of Technology Indore,
Simrol, Khandwa Road, Indore 453 552, India
}
\email{vibhutiarora1991@gmail.com}

\author[S. K. Sahoo]{Swadesh Kumar Sahoo${}^*$}
\address{Swadesh Kumar Sahoo, Discipline of Mathematics,
Indian Institute of Technology Indore,
Simrol, Khandwa Road, Indore 453 552, India}
\email{swadesh@iiti.ac.in}

\thanks{${}^*$ The corresponding author}

\begin{abstract}
We consider the family of all meromorphic functions $f$ of the form 
$$
f(z)=\frac{1}{z}+b_0+b_1z+b_2z^2+\cdots
$$ 
analytic and locally univalent in the puncture disk $\mathbb{D}_0:=\{z\in\mathbb{C}:\,0<|z|<1\}$. 
Our first objective in this paper is to find a sufficient condition  
for $f$ to be meromorphically convex of order $\alpha$, $0\le \alpha<1$, 
in terms of the fact that the absolute value of the well-known Schwarzian derivative 
$S_f (z)$ of $f$ is bounded above by a smallest positive root of a non-linear equation. 
Secondly, we consider a family of functions $g$ of the form 
$g(z)=z+a_2z^2+a_3z^3+\cdots$ analytic and locally univalent in the open unit disk 
$\mathbb{D}:=\{z\in\mathbb{C}:\,|z|<1\}$, and 
show that $g$ is belonging to a family of functions convex in one direction if
$|S_g(z)|$ is bounded above by a small positive constant depending on the second coefficient $a_2$. 
In particular, we show that such functions $g$ are also contained in the starlike and close-to-convex family.
\\


\end{abstract}

\maketitle
\thispagestyle{empty}


\section{Introduction}\label{sec1}
Recall that a function $f$ which is analytic in a region, except possibly at poles, is said to be meromorphic in that region.
Hence, analytic functions are by default meromorphic without poles.
In this paper, we consider the family of all meromorphic functions $f$ of the form 
$$
f(z)=\frac{1}{z}+b_0+b_1z+b_2z^2+\cdots
$$
defined in the open unit disk $\mathbb{D}:=\{z\in\mathbb{C}:\,|z|<1\}$. Clearly, $f$ has a simple pole 
at the origin, and hence it is analytic in the puncture disk $\mathbb{D}_0:=\{z\in\mathbb{C}:\,0<|z|<1\}$. 
Let us denote this family of meromorphic functions by $\mathcal{B}$. The
set of all univalent functions in $\mathcal{B}$ is usually denoted by $\Sigma$.
We also consider the family $\mathcal{A}$, of functions $g$ analytic in $\mathbb{D}$ of the form
$$
g(z)=z+a_2z^2+a_3z^3+\cdots.
$$
A quick observation which can easily be verified that 
\begin{equation}\label{fg-eqn}
f\in \mathcal{B} \iff g=1/f\in \mathcal{A}.
\end{equation}

A single valued function $f$ is said to be {\em univalent} 
(or schlicht) in a domain $D\subset \mathbb{C}$ if it never takes the same value twice: $f(z_1)\neq f(z_2)$ for all 
$z_1 \neq z_2$ in $D$. The family of all univalent functions $g\in \mathcal{A}$ is denoted by $\mathcal{S}$. 
Such functions $g$ are of interest because they appear in the Riemann mapping theorem.
The study of the family $\mathcal{S}$ became popular when the Bieberbach conjecture was
first posed in 1916 and remained as a challenge to all mathematicians until 1985 when it was solved by
de Branges. Since then the conjecture is known as the de Branges Theorem. 
This problem has been attracted to many mathematicians in introducing certain subclasses of $\mathcal{S}$
and developing important new methods in geometric function theory.
The de Branges theorem gives a necessary condition for a function $g$ to be in $\mathcal{S}$ in terms of
its Taylor's coefficient. On the other hand, several important sufficient conditions for functions to be in $\mathcal{S}$ 
were also introduced by several researchers to generate its subclasses having interesting geometric properties.
Part of this development is the family of convex functions, starlike functions, close-to-convex functions, etc.
Later, counterpart of this development for the family $\Sigma$ of 
meromorphic univalent functions were also studied extensively.   
We refer to the standard books by Duren \cite{Dur83}, Goodman \cite{Goo83}, Lehto \cite{Leh87},
and Pommerenke \cite{Pom75} for the literature on the topic. Therefore, the study of 
sufficient conditions for functions to be in $\mathcal{S}$, in particular, in its subfamilies are important
in this context. In this paper, we mainly deal with such properties in terms of the well-known Schwarzian 
derivative of locally univalent functions. 

First let us recall the definition of the Schwarzian derivative. 
Let $f$ be a {\em meromorphic function} and $f'(z) \neq 0$ in $\mathbb{D}:=\{z\in \mathbb{C}: |z|<1\}$ (in other words, 
we say, $f$ is locally univalent in $\mathbb{D}$), then the Schwarzian derivative of $f$ at $z$ is defined as
$$
S_f (z)=\left(\cfrac{f''}{f'}\right)'-\cfrac{1}{2} \left(\cfrac{f''}{f'}\right)^2.
$$
It is appropriate here to recall from texts that $S_f=0$ if and only if $f$ is a M\"obius
transformation (see for instance, \cite[p~51]{Leh87}).
A simple computation through \eqref{fg-eqn} yields the useful relation
$$
S_f(z)=S_g(z)
$$ 
for all locally univalent meromorphic functions $f\in\mathcal{B}$ and $g=1/f\in\mathcal{A}$. 

The study of necessary and sufficient conditions for functions to be univalent, in particular to be 
starlike, convex, close-to-convex, in terms of Schwarzian derivatives are attracted by a number of mathematicians. 
It is a surprising fact is that most of such necessary conditions are proved using standard theorems in complex
variables, whereas sufficient conditions are proved through initial value problems 
of differential equations; see for instance \cite{Dur83,Leh87}. The conditions of the form
\begin{equation}\label{eqn:1}
|S_f(z)| \leq \cfrac{C_0}{(1-|z|^2)^2},
\end{equation}
for a positive constant $C_0$, have been most popular to many mathematicians. For instance, Nehari in 1949 first proved that 
if $g$ is an analytic and locally univalent function in $\mathbb{D}$ satisfying
\eqref{eqn:1} with $C_0=2$
then $g$ is univalent in $\mathbb{D}$. This condition becomes necessary 
when the constant $C_0=6$; see \cite{Neh49}. 
Hille \cite{Hil49} showed that the constant $2$ in the sufficient condition of Nehari is the 
best possible constant. 
Related problems are also investigated in \cite{Neh54,Pok51,Neh79}.
Thus, applications of the Schwarzian derivative can be seen in second order linear differential equations, univalent functions,
and also in Teichm{\"u}ller spaces \cite{Dur83,Pom75}. Note that if $g\in\mathcal{A}$ is univalent 
then \eqref{fg-eqn} leads to the useful coefficient relation $|a_2^2-a_3|=|S_f(0)|/6$; see \cite[p.~263]{Dur83}.

Another form of sufficient condition for univalency in terms of Schwarzian derivative attracted by
many researchers in this field is 
\begin{equation}\label{eqn:2}
|S_g(z)| \leq 2C_1,
\end{equation}
for some positive constant $C_1$.
If $g\in \mathcal{A}$ satisfies \eqref{eqn:2} with $C_1=\pi^2/4$, then it is proved by Nehari \cite{Neh49}
that $g$ is univalent in $\mathbb{D}$.
Gabriel \cite{Gab55} studied a sufficient condition for a function $g\in \mathcal{A}$ to be starlike 
in the form \eqref{eqn:2} for some optimal constant $C_1$.
Sufficient condition in the form \eqref{eqn:2} for convexity of order $\alpha$ 
is investigated by Chiang in \cite{Chi94}. However, the best possible constant is not yet known in this case. 
Kim and Sugawa in \cite{KS} obtained the 
sufficient condition in the form \eqref{eqn:2} for starlikeness of order $\alpha$ 
by fixing the second coefficient of the function.
 
Our main objective in this paper is to study the sufficient conditions of the form \eqref{eqn:2} 
for meromorphically convex functions of order $\alpha$ and 
for functions in a family that are convex in one direction, in particular in the starlike and close-to-convex family.
%
%
Rest of the structure of this paper is as follows. Section \ref{sec2} is devoted to the definitions of the classes of functions
and statements of our main results. Section \ref{sec3} deals with some preliminary results those are used to prove our
main results. Finally, the proof of our main results are given in Section \ref{sec4}
followed by examples of functions satisfying these results.


\section{Definitions and main results}\label{sec2}
 
This section is divided into two subsections. The first subsection concerns about the definition of a 
subclass of the class $\mathcal{B}$, 
namely, the meromorphically convex functions of order $\alpha$ having simple pole at 
$z=0$, and the main results associated with these functions. The second subsection 
deals with some well-known analytic functions convex in one direction, 
in particular, functions in the starlike and close-to-convex 
families. Sufficient conditions 
in the form \eqref{eqn:2}
for functions to be in these families are also stated.

\subsection{Meromorphic functions in $\mathbb{D}$ with a simple pole at $z=0$}\label{ssec1}
If $f \in \mathcal{B}$ satisfies $f(z) \neq 0$ in $\mathbb{D}_0$ and 
$$
-{\rm Re}\Big(\frac{zf'(z)}{f(z)}\Big)>\alpha \quad(z \in \mathbb{D},\, 0 \leq \alpha<1),
$$
then $f$ is said to be {\em meromorphically starlike of order $\alpha$}. 
A function $f\in\mathcal{B}$ is said to be {\em meromorphically starlike (of order $0$)}
if and only if complement of $f(\mathbb{D}_0)$ is starlike 
with respect to the origin (see \cite[p.~265, Vol. 2]{Goo83}). Note that
meromorphically starlike functions are univalent and hence they lie on the class
$\Sigma$.
Similarly, if $f \in \mathcal{B}$ satisfies $f(z) \neq 0$ in $\mathbb{D}_0$ and 
\begin{equation}\label{e:3}
-{\rm Re}\Big(1+\frac{zf''(z)}{f'(z)}\Big)>\alpha \quad(z \in \mathbb{D},\, 0 \leq \alpha<1),
\end{equation}
then $f$ is said to be {\em meromorphically convex of order $\alpha$}. 
If $\alpha=0$, the inequality (\ref{e:3}) is equivalent to the definition of meromorphically convex functions. That is,
$f$  maps $\mathbb{D}$ onto the complement of a convex region \cite{Gab55,Nun00}. 
In this case, we say $f$ is {\em meromorphically convex}. Note that 
meromorphically convex functions are also univalent and hence they lie in the class
$\Sigma$.
For more geometric properties of these classes, we refer to the standard books 
\cite{Goo83, Mimo00}.

%
%
%
Main results of this paper deal with functions whose Schwarzian derivatives are 
bounded above by some constant, that is, functions satisfy (\ref{eqn:2}).
Note that if $S_f(z)$ is uniformly bounded in $\mathbb{C}$, then the Schwarzian derivative is still well defined.
Hence the assumption that $f$ is locally univalent at a point $z$ (or $g'(z) \neq 0$), 
in (\ref{eqn:2}) is not chosen; see also Tichmarsh \cite[p.~198]{Tic39}.

Gabriel modified Nehari's technique to show univalency and convexity property of
functions $f\in\mathcal{B}$ and proved the following:

\medskip
\noindent{\bf Theorem A.} \cite[Theorem 1]{Gab55}
{\em If $f\in\mathcal{B}$ satisfies 
\begin{equation}\label{ThmAeq1}
|S_f (z)| \leq 2 c_0 \quad \mbox{ for }|z|<1,
\end{equation}
where $c_0$ is the smallest positive root of the equation 
$$
2\sqrt{x}-\tan \sqrt{x}=0,
$$
then $f$ is univalent in the punctured disk and maps the interior of 
each circle $|z|=r<1$ onto the exterior of a convex region. The constant $c_0$ is the largest possible constant
satisfying \eqref{ThmAeq1}.}

An analog to this result for meromorphically convex functions of order 
$\alpha$ is one of our main results which is stated below.

\begin{theorem}\label{main-thm1}
Let $0 \leq \alpha <1$.
If $f\in\mathcal{B}$ satisfies 
\begin{equation}\label{eq:1}
|S_f (z)| \leq 2 c_\alpha \quad \mbox{ for $|z|<1,$}
\end{equation}
where $c_\alpha$ is the smallest positive root of the equation
\begin{equation}\label{eq:2}
2\sqrt{x}-(1+\alpha) \tan \sqrt{x}=0
\end{equation}
depending on $\alpha$, then 
\renewcommand{\theenumi}{\alph{enumi}}
\begin{enumerate}
\item $f$ is meromorphically convex of order $\alpha$; and
\item the quantity $c_\alpha$ is the largest possible constant satisfying \eqref{eq:1}.
\end{enumerate}
\end{theorem}

In particular, if $\alpha =0$, Theorem \ref{main-thm1} reduces to Theorem A.  
\subsection{Analytic functions in $\mathbb{D}$}\label{ssec2}
A function $g\in\mathcal{A}$ is said to be {\em convex of order $\beta$}, $0\le \beta<1$, if and only if 
$$
{\rm Re}\,\Big(1+\frac{zg''(z)}{g'(z)}\Big)>\beta,\quad z\in\mathbb{D}.
$$
Chiang proved the following sufficient condition for convex functions of order $\beta$ in terms of small 
Schwarzian derivative:

\medskip
\noindent
{\bf Theorem B.} \cite[Theorem 2]{Chi94}
{\em Let $g \in \mathcal{A}$ and $|a_2|=\eta<1/3$. Suppose that
$$
\sup _{z \in \mathbb{D}}|S_g(z)|=2\delta,
$$
where $\delta=\delta(\eta)$ satisfies the inequality
$$
6\eta+5\delta(1+\eta)e^{\delta/2}<2.
$$
Then $g$ is convex of order
$$
\cfrac{2-6\eta-5(1+\eta)\delta e^{\delta/2}}{2-2\eta-(1+\eta)\delta e^{\delta/2}}\,.
$$
}

\medskip
Our aim in this subsection is to state results similar to Theorem B for certain functions convex
in one direction, in particular, for functions in the family of starlike and close-to-convex functions. 

For $\beta \geq 3/2$, we consider the class $\mathcal{C}_{\beta}$ introduced by Shah in \cite{SHAH73} as follows:
$$
\mathcal{C}_{\beta}=\left\{g\in \mathcal{A}:\,\frac{-\beta}{2\beta-3}
<{\rm Re}\Big(1+ \frac{zg''(z)}{g'(z)}\Big)<\beta, ~z \in \mathbb{D} \right\}.
$$
This originally follows from a sufficient condition for a function $g$ to be convex in one direction
studied by Umeraza in \cite{UME52}. 
Note that the special cases $\mathcal{C}_{3/2}$ and $\mathcal{C}_\infty$ are contained in the family of starlike and close-to-convex functions respectively 
(see the detailed discussion below in this section). It is a natural question to ask for
functions belonging to the family $\mathcal{C}_\beta$ for all $\beta\ge 3/2$. Such functions
can be generated in view of \cite[Theorem~12]{SHAH73}, which says that for all functions 
$f\in\mathcal{A}$ satisfying 
$$
\frac{\beta}{3-2\beta}<{\rm Re}\Big(\frac{zf'(z)}{f(z)}\Big)<\beta,
$$
the Alexander transform of $f$ belongs to the family $\mathcal{C}_\beta$, $\beta\ge 3/2$.

We now state our second main result which provides a sufficient condition for functions to be in 
$\mathcal{C}_\beta$ with respect to its small Schwarzian derivative.

\begin{theorem}\label{main-thm2}
For $\beta \geq 3/2,$ set 
$$
\phi(\beta)=\min \left\{\cfrac{\beta-1}{\beta+1}\,,\cfrac{6(\beta-1)}{2(7\beta-9)}\right\}
~~\mbox{ and }~~
\psi(\beta)=\max \left\{\cfrac{\beta+3}{\beta+1}\,,\cfrac{11\beta-15}{7\beta-9}\right\}.
$$
Let $g \in \mathcal{A}$ and $|a_2|=\eta<\phi(\beta)$. Suppose that
$$
\sup _{z \in \mathbb{D}}|S_g(z)|=2\delta,
$$
where $\delta=\delta(\eta)$ satisfies the inequality
\begin{equation}\label{e:4}
2\eta+ \psi(\beta)\delta(1+\eta)e^{\delta/2}<2\phi(\beta).
\end{equation}
Then $g \in \mathcal{C}_{\beta}$. In particular, $g$ is convex in one direction.
\end{theorem}

A function $g\in\mathcal{A}$ is said to be {\em starlike of order $\beta$}, $0\le \beta<1$, if and only if 
$$
{\rm Re}\,\Big(\frac{zg'(z)}{g(z)}\Big)>\beta,\quad z\in\mathbb{D}.
$$
In particular, for $\beta=0$, we simply call such functions $g$ as {\em starlike functions}.
Recall the sufficient condition for starlike functions $g\in\mathcal{A}$ from \cite[(16)]{PR95} which tells us that
$$
{\rm Re}\Big(1+ \frac{zg''(z)}{g'(z)}\Big)<\frac{3}{2}\, \implies \Big|\frac{zg'(z)}{g(z)}-\frac{2}{3}\Big|<\frac{2}{3}.
$$
This generates the following subclass of the class of starlike functions:
$$
\mathcal{C}_{3/2}:=\left\{g\in \mathcal{A}:\,{\rm Re}\Big(1+ \frac{zg''(z)}{g'(z)}\Big)
<\frac{3}{2}\right\}.
$$
This particular class of functions is also studied in different contexts in \cite{PS08}.

The following corollary immediately follows from Theorem~\ref{main-thm2} for the class $\mathcal{C}_{3/2}$.




\begin{corollary}\label{corollary1}
Let $g \in \mathcal{A}$ and $|a_2|=\eta<1/5.$ Suppose 
$$
\sup _{z \in \mathbb{D}}|S_g(z)|=2\delta
$$
where $\delta=\delta(\eta)$ satisfies the inequality
\begin{equation}\label{eq:22}
10\eta+9\delta(1+\eta)e^{\delta/2}<2.
\end{equation}
Then $g\in \mathcal{C}_{3/2}$. In particular, $g$ is starlike.
\end{corollary}

We next recall what is close-to-convex function
followed by a subclass of the class of close-to-convex functions and then state the corresponding result which is again
an easy consequence of Theorem~\ref{main-thm2}.

We here adopt the well-known Kaplan characterization for close-to-convex functions. Let $g\in\mathcal{A}$ be 
locally univalent. Then $g$ is {\em close-to-convex} if and only if 
$$
\int_{\theta_1}^{\theta_2} {\rm Re}\Big(1+\frac{zg''(z)}{g'(z)}\Big)\,d\theta>-\pi,\quad z=re^{i\theta},
$$
for each $r~(0<r<1)$ and for each pair of real numbers $\theta_1$ and $\theta_2$ with $\theta_1<\theta_2$.
If a locally univalent analytic function $g$ defined in $\mathbb{D}$ satisfies 
$$
{\rm Re}\Big(1+\frac{zg''(z)}{g'(z)}\Big)>-1/2,
$$
then by the Kaplan characterization it follows easily that $g$ is close-to-convex in $\mathbb{D}$ (here 
$\theta_1$ and $\theta_2$ are chosen as $0$ and $2\pi$ respectively) and hence $g$ is univalent in $\mathbb{D}$.
This generates the following subclass of the class of close-to-convex (univalent) functions:
$$
\mathcal{C}_\infty:=\left\{g\in \mathcal{A}:\,{\rm Re}\Big(1+ \frac{zg''(z)}{g'(z)}\Big)>-\frac{1}{2}\right\}.
$$
This class of functions is also studied recently by several authors in different contexts; 
for instance see \cite{BP14,LP17,MYLP14,PSY14}
and references therein.

Now we are ready to state our sufficient condition for functions $g$ to be in 
$\mathcal{C}_\infty$ in terms of their Schwarzian derivatives bounded by small quantity.

\begin{corollary}\label{corollary2}
Let $g\in \mathcal{A}$ and $|a_2|=\eta<3/7.$ Suppose that
$$
\sup _{z \in \mathbb{D}}|S_g(z)|=2\delta
$$
where $\delta=\delta(\eta)$ satisfies the inequality
\begin{equation}\label{eq:250}
14\eta+11\delta(1+\eta)e^{\delta/2}<6.
\end{equation}
Then $g\in \mathcal{C}_\infty$
and hence $g$ is close-to-convex function.
\end{corollary}


\section{Preliminary results}\label{sec3}

\subsection*{Connection with a linear differential equation}

In this section we study a relationship between Schwarzian derivative of a meromorphic function $f$ 
and solution of a second order linear differential equation depending on $f$.
 
Recall the following lemma from Duren \cite[p. 259]{Dur83}. 
\begin{lemma}\label{l:-1}
For a given analytic function $p(z)$, a meromorphic function $f$ has the Schwarzian derivative
of the form 
$S_f(z)=2p(z)$ if and only if $f(z)=w_1(z)/w_2(z)$ for any pair of linearly independent solutions 
$w_1(z)$ and $w_2(z)$ of the linear differential equation
\begin{equation}\label{eq:3}
w''+p(z)w=0.
\end{equation}
\end{lemma}
Note that an example satisfying Lemma~\ref{l:-1} is described in the proof of 
Theorem~\ref{main-thm1}(b).
Assume now that $w_1(z)$ and $w_2(z)$ satisfy the following conditions:
\begin{align*}
w_1(0)=1, & ~w_2(0)=0;\\
w_1'(0)=0, & ~w_2'(0)=1.
\end{align*}
Clearly $w_1(0)$ and $w_2(0)$ are linearly independent since the Wronskian 
$W(w_1(0),w_2(0))$ is non-vanishing.
Recall that
\begin{equation} \label{eq:6}
f(z)=\frac{w_1(z)}{w_2(z)} =\cfrac{1}{z}+b_0+b_1z+\cdots.
\end{equation}
Hence, a simple computation on logarithmic derivative of $f'(z)$ leads to
$$
\cfrac{f''(z)}{f'(z)}=\frac{w_2(z)w_1''(z)-w_1(z)w_2''(z)}{w_2(z)w_1'(z)-w_1(z)w_2'(z)}-2\frac{w_2'(z)}{w_2(z)}.
$$
Since $w_1(z)$ and $w_2(z)$ satisfy \eqref{eq:3}, it follows that
$$
\cfrac{f''(z)}{f'(z)}=-2\frac{w_2'(z)}{w_2(z)},
$$
and hence we have the relation
\begin{equation}\label{eq:4}
{\rm Re}\Big(1+\frac{zf''(z)}{f'(z)}\Big) 
= 1-2{\rm Re}\Big(\frac{zw_2'(z)}{w_2(z)}\Big).
\end{equation}


\subsection*{The Function $\boldsymbol{2x-(1+\alpha)\tan x}$.}
For $0\le \alpha<1$, we set
$$
h(x):=2x-(1+\alpha)\tan x.
$$
Derivative test for $h(x)$ tells us that 
$h(x)$ is decreasing 
in $(\arctan (\sqrt{(1-\alpha)/(1+\alpha)}) ,\pi/2)$.

Then the following lemma is useful.

\begin{lemma}\label{l:0}
Let $\beta<\pi/2$ be the smallest positive root of  $h(x)=2x-(1+\alpha)\tan x=0$ for some $\alpha>0$. Then 
$$
\beta\ge \arctan\sqrt{(1-\alpha)/(1+\alpha)}
$$
holds true.
\end{lemma}
\begin{proof}
Given that $h(\beta)=0=2\beta-(1+\alpha)\tan \beta$. This gives
\begin{equation}\label{lm1}
\alpha=\cfrac{2\beta}{\tan \beta}-1.
\end{equation}

On contrary, suppose that $0< \beta < \arctan\sqrt{(1-\alpha)/(1+\alpha)}<\pi/2$. This implies that 
$$
\tan^2 \beta < \cfrac{1-\alpha}{1+\alpha}.
$$
Substituting the value of $\alpha$ in (\ref{lm1}), we obtain
$$
\tan^2 \beta < \cfrac{\tan \beta}{\beta}-1
$$
equivalently,
$$
\sec^2 \beta < \cfrac{\tan \beta}{\beta} \iff 2\beta < \sin 2\beta,
$$
which is a contradiction. Thus, the proof of our lemma is complete.
\end{proof}

Let $c_\alpha$ be the smallest positive root of the equation (\ref{eq:2}).
Since $h(\sqrt{c_\alpha})=0$, it follows by Lemma~\ref{l:0} that
\begin{equation}\label{eq:25}
 h(x)  \left\{ \begin{array}{ll}
         \geq 0, & \mbox{for $0 \leq x \leq \sqrt{c_\alpha}$};\\
        < 0, & \mbox{for $\sqrt{c_\alpha}<x < \pi/2$}.\end{array} \right. 
\end{equation}
If we replace $x$ by $x\sqrt{c}$, $c>0,$ in (\ref{eq:25}), we obtain
\begin{equation} \label{eq:160}
h(x\sqrt{c})=2x\sqrt{c}-(1+\alpha)\tan(x\sqrt{c}) \geq 0 \mbox{ for $0 \leq x\sqrt{c} \leq \sqrt{c_\alpha}$} 
\end{equation}
and
\begin{equation} \label{eq:17}
h(x\sqrt{c})=2x\sqrt{c}-(1+\alpha)\tan(x\sqrt{c}) < 0 \mbox{ for $\sqrt{c_\alpha}<x\sqrt{c}<\pi/2$}. 
\end{equation}
We may have the following two cases when $h(x\sqrt{c})$ is negative.

{\bf Case 1:}
If $c \leq c_\alpha$, then \eqref{eq:17} gives that $h(x\sqrt{c})$ is also negative in  $[1,\pi/2\sqrt{c})$.

{\bf Case 2:} 
If $c > c_\alpha$, then \eqref{eq:17} gives that $h(x\sqrt{c})$ is also negative in $(\sqrt{c_\alpha/c},1).$


In the sequel, we collect the following lemmas to be used in the proof of Theorem~\ref{main-thm1}.
\begin{lemma}\label{l:1}
A function $f\in\mathcal{B}$ in the form \eqref{eq:6} is meromorphically convex of order $\alpha$ 
if and only if $w_2(z)$ is starlike of order $(\alpha+1)/2$.  
\end{lemma}
\begin{proof}
Condition (\ref{eq:4}) is equivalent to
$$
-{\rm Re}\Big(1+\frac{zf''(z)}{f'(z)}\Big) =  -1+2{\rm Re}\Big(\frac{zw_2'(z)}{w_2(z)}\Big),
$$
which yields
$$
-{\rm Re}\Big(1+\frac{zf''(z)}{f'(z)}\Big) > \alpha \iff {\rm Re}\Big(\frac{zw_2'(z)}{w_2(z)}\Big)> \frac{\alpha+1}{2}~.
$$
Since $w_2(0)=0$ and $w_2'(0)=1$, $w_2(z)$ is starlike of order $(\alpha+1)/2$.  Thus, completing the proof of our lemma.
\end{proof}

\begin{remark}
A simple computation using the identity (\ref{eq:4}) yields
$$
{\rm Re}\,\Big(\frac{zw_1'(z)}{w_1(z)}\Big)=\frac{1}{2}+{\rm Re}\,\Big(\frac{zf'(z)}{f(z)}\Big)-\frac{1}{2}{\rm Re}\,\Big(1+\frac{zf''(z)}{f'(z)}\Big).
$$
Therefore, the function $w_1$ is not necessarily starlike when the function
$f$ is meromorphically convex.
\end{remark}


\begin{lemma}\label{l:2}
For $0 \leq \alpha <1 $, let $c_\alpha$ ~$(0<c\leq c_\alpha)$ be the root of the equation given by $(\ref{eq:2})$. Then we have
\begin{equation}\label{eq:11}
{\rm Re}(z\sqrt{c}\cot(z\sqrt{c})) > \frac{\alpha+1}{2}, \quad |z|<1.
\end{equation}
\end{lemma}
\begin{proof}
Substituting $z=x+iy$ in (\ref{eq:11}), we see that the desired inequality is equivalent to
$$
2{\rm Re}\Big(\sqrt{c}(x+iy)\cfrac{\cos(\sqrt{c}(x+iy))}{\sin(\sqrt{c}(x+iy))}\Big) > \alpha+1.
$$
This is, using the basic identities $2{\rm Re}\,w=w+\overline{w}$, $\cos (iy)=\cosh (y)$, and $\sin (iy)=i\sinh (y)$, we see that it is equivalent to proving
\begin{align*}
2x\sqrt{c}\sin (\sqrt{c}x) \cos (\sqrt{c}x) 
&+2 y\sqrt{c} \sinh (\sqrt{c}y) \cosh (\sqrt{c}y)\\
& > (1+\alpha)(\sin^2(\sqrt{c}x)+\sinh^2(\sqrt{c}y)).
\end{align*}
So, it suffices to prove the inequality 
\begin{equation}\label{eq:12}
\begin{array}{lr}
\sin(\sqrt{c}x) \cos(\sqrt{c}x) [2\sqrt{c}x-(1+\alpha)\tan(\sqrt{c}x)]\\
\hspace*{4cm}>  \sinh(\sqrt{c}y) \cosh(\sqrt{c}y)[(1+\alpha)\tanh (\sqrt{c}y)-2y\sqrt{c}y]
\end{array}
\end{equation}
for $0<c\leq c_\alpha$ and $x^2+y^2<1.$ First consider the points $x,~y$ in the first quadrant. Then we see that
$\sin(\sqrt{c}x)$, $ \cos(\sqrt{c}x)$, $\sinh(\sqrt{c}y)$ and $\cosh(\sqrt{c}y)$ are all positive since $c<c_\alpha<\pi^2/4$.
Also $2x\sqrt{c}-(1+\alpha)\tan(\sqrt{c}x)$ is positive which follows from (\ref{eq:160}).
On the other hand, $(1+\alpha)\tanh (\sqrt{c}y)-2(\sqrt{c}y)$ 
is non-positive because
$g(y)=(1+\alpha)\tanh (\sqrt{c}y)-2(\sqrt{c}y)$ is decreasing, hence for $y \geq 0$ we obtain
$$
g(y)=(1+\alpha)\tanh (\sqrt{c}y)-2\sqrt{c}y \leq 0.
$$
Hence, the inequality (\ref{eq:12}) holds true in the first quadrant. Now if we replace $x$ by $-x$ and $y$ by $-y$ then the inequality (\ref{eq:12}) remains same in all the other
quadrants of $\mathbb{D}$. The desired inequality thus follows.
\end{proof}

The following results of Gabriel are also useful.


\begin{lemma}\cite[Lemma 4.1]{Gab55} \label{l:3}
If $w(z)$ satisfies $(\ref{eq:3})$ with $w(0)=0$ and $w'(0)=1,$ then for $0<\rho\le r<1$ and for a fixed $\theta\in [0,2\pi]$, we have
\begin{equation}\label{eq:7}
|w(re^{i\theta})|^2 {\rm Re}\Big(\frac{re^{i\theta}w'(re^{i\theta})}{w(re^{i\theta})}\Big)= r \int_0 ^r |w'(\rho e^{\theta})|^2 d\rho -r \int_0 ^r {\rm Re}(\rho^2e^{2i\theta}p(\rho e^{i\theta})) 
\cfrac{|w(\rho e^{i\theta})|^2}{\rho^2} d\rho.
\end{equation}
\end{lemma}


\begin{lemma}\cite[Lemma 4.2]{Gab55} \label{l:4}
Let $y(\rho)$ and $y'(\rho)$ be continuous real functions of $\rho$ for $0\leq \rho < 1.$ For small values of $\rho$ let
$y(\rho)=O(\rho).$ Then 
\begin{equation}\label{eq:8}
r \int_0 ^r [y'(\rho)]^2 d\rho-cr  \int_0 ^r [y^2(\rho)] d\rho -r\sqrt{c} \cot (r\sqrt{c})\cdot y^2(r) \geq 0
\end{equation}
for $0<r<1$ and $c>0.$ Equality holds for 
$$
y(\rho)=c^{-1/2} \sin (\rho\sqrt{c}), \quad c>0.
$$
\end{lemma}

\section{Proof of the main results}\label{sec4}

\subsection{Proof of Theorem \ref{main-thm1}}
Given that $f\in\mathcal{B}$ satisfies \eqref{eq:1}
and $c_\alpha$ is the smallest positive root of the equation \eqref{eq:2}.
A simple computation yields
$$
\alpha=\cfrac{2\sqrt{c_{\alpha}}-\tan \sqrt{c_{\alpha}}}{\tan \sqrt{c_{\alpha}}}.
$$
Differentiating $\alpha$ with respect to $c_{\alpha}$, we obtain
$$
\cfrac{d\alpha}{d c_{\alpha}}=\cfrac{\tan \sqrt{c_{\alpha}}-\sqrt{c_{\alpha}}\sec^2 \sqrt{c_{\alpha}}}
{\sqrt{c_{\alpha}}\tan^2 \sqrt{c_{\alpha}}}.
$$
Since $\tan x-x\sec^2 x\le 0$ is equivalent to $\sin 2x\le 2x$, which is always true for all $x\in \mathbb{R}$,
it follows that $c_{\alpha}$ increases if and only if $\alpha$ decreases.
 

Now we proceed for completing the proof of (a) and (b).
\renewcommand{\theenumi}{\alph{enumi}}
\begin{enumerate}
\item In this part we prove that $f$ is meromorphically convex of order $\alpha$, $0\le \alpha<1$, that is
$f$ satisfies \eqref{e:3}. 

Set $S_f(z)=2p(z)$ for a given analytic function $p(z)$. Then by \eqref{eq:1}, it follows that $|p(z)|\leq c_\alpha,$
and hence we have
$$
{\rm Re}(z^2 p(z)) \leq c_\alpha |z|^2 \quad \mbox{for $|z|<1.$}
$$
By Lemma~\ref{l:-1}, the function has the form $f(z)=w_1(z)/w_2(z)$ for any pair of linearly independent solutions 
$w_1(z)$ and $w_2(z)$ of the linear differential equation \eqref{eq:3}.
Clearly, the particular solution $w_2(z)$ satisfies the hypothesis of Lemma \ref{l:3}. 
Since ${\rm Re}(z^2 p(z)) \leq c_\alpha |z|^2$ holds, (\ref{eq:7}) implies
\begin{equation}\label{eq:9}
|w_2(re^{i\theta})|^2{\rm Re}\Big(\frac{re^{i\theta}w_2'(re^{i\theta})}{w_2(re^{i\theta})}\Big)\geq r\int_0 ^r |w_2'(\rho e^{i\theta})|^2 d\rho -r c_\alpha \int_0 ^r{|w_2(\rho e^{i\theta})|^2} d\rho,
\end{equation}
for $0<\rho\le r<1$ and for some fixed $\theta$.

Putting $w_2(\rho e^{i\theta})=u_2(\rho,\theta)+i v_2(\rho,\theta)$. For a constant ray  
$\theta$, $w_2$ will become a function of $\rho$ only. Note that $u_2(\rho)$ and $v_2(\rho)$ satisfies the hypothesis of Lemma \ref{l:4}. We obtain the following two inequalities
after substituting  $u_2(\rho)$ and $v_2(\rho)$ in (\ref{eq:8}) and replacing $c$ by $c_\alpha$
\begin{equation}\label{l:18}
r \int_0 ^r [u_2'(\rho)]^2 d\rho-c_\alpha r  \int_0 ^r [u_2^2(\rho)] d\rho -\sqrt{c_\alpha}r \cot (\sqrt{c_\alpha}r)
\cdot u_2^2(r) \geq 0, 
\end{equation}
and 
\begin{equation}\label{l:19}
r \int_0 ^r [v_2'(\rho)]^2 d\rho-c_\alpha r  \int_0 ^r [v_2^2(\rho)] d\rho -\sqrt{c_\alpha}r 
\cot (\sqrt{c_\alpha}r)\cdot v_2^2(r) \geq 0.
\end{equation}

Since $w_2(\rho e^{i\theta})=u_2(\rho,\theta)+i v_2(\rho,\theta)$, addition of (\ref{l:18}) and (\ref{l:19}) leads to
\begin{equation}\label{eq:10}
r \int_0 ^r |w_2'(\rho e^{i\theta})|^2 d\rho -r c_\alpha\int_0 ^r {|w_2(\rho e^{i\theta})|^2} d\rho \geq \sqrt{c_\alpha} r \cot (\sqrt{c_\alpha}r)
|w_2|^2.
\end{equation}

Comparing (\ref{eq:9}) with (\ref{eq:10}), we obtain
$$
|w_2(re^{i\theta})|^2 {\rm Re}\Big(\cfrac{zw_2'(re^{i\theta})}{w_2(re^{i\theta})}\Big)\geq \sqrt{c_\alpha} r \cot (\sqrt{c_\alpha}r)|w_2(re^{i\theta})|^2,
$$
that is,
\begin{equation}\label{eq:16}
{\rm Re}\Big(\cfrac{zw_2'(z)}{w_2(z)}\Big)\geq \sqrt{c_\alpha} r \cot (\sqrt{c_\alpha}r) \quad \mbox{for $|z|=r<1.$}
\end{equation}

It follows from Lemma \ref{l:2} that
\begin{equation}\label{eq:13}
\sqrt{c_\alpha}  r \cot (\sqrt{c_\alpha}r)={\rm Re}(\sqrt{c_\alpha} r \cot (\sqrt{c_\alpha}r)) > 
\cfrac{\alpha+1}{2}.
\end{equation}
Comparison of (\ref{eq:16}) with (\ref{eq:13}) yields
$$
{\rm Re}\Big(\cfrac{zw_2'(z)}{w_2(z)}\Big) > \cfrac{\alpha+1}{2},
$$
and hence it follows from Lemma \ref{l:1} that $f$ is meromorphically convex of order 
$\alpha$.


\item We prove that the quantity $c_\alpha$ is the largest possible
constant satisfying \eqref{eq:1}, i.e. we can not replace $c_\alpha$
by a larger quantity. We prove this by contradiction. 
If we replace 
$c_\alpha$ by a larger number $c=c_\alpha+\epsilon$ for some $\epsilon>0,$ 
then we observe that there exists a function $f\in\mathcal{B}$ 
satisfying
\begin{equation}
|S_f (z)| \leq 2 (c_\alpha+\epsilon), \quad \mbox{$|z|<1,$}
\end{equation}
but $f$ is not meromorphically convex of order $\alpha$. For this, we
consider the function
$$
f(z)=\cfrac{w_1(z)}{w_2(z)}, \quad |z|<1,
$$
with the two linearly independent solutions
$$
w_1(z)=\cos (\sqrt{c}z)~~\mbox{ and }~~
w_2(z)=\cfrac{\sin(\sqrt{c}z)}{\sqrt{c}}
$$
of the differential equation $w''+cw=0$. Clearly, by a simple computation,
the function $f(z)=\sqrt{c}\cot(\sqrt{c}z)$ satisfies $S_f (z)=2c$.
It remains to show that this function $f$ is not meromorphically convex
of order $\alpha$, equivalently, by definition, we prove that 
$$
-{\rm Re}\Big(1+\cfrac{z_0f''(z_0)}{f'(z_0)}\Big)\le \alpha
$$
for some $z_0\in\mathbb{D}$.
By Lemma~\ref{l:1}, it is equivalently to proving
\begin{equation}\label{eq:14}
{\rm Re}\Big(\cfrac{z_0w'_2(z_0)}{w_2(z_0)}\Big)
= {\rm Re}\Big(\cfrac{\sqrt{c}z_0\cos(\sqrt{c}z_0)}{\sin(\sqrt{c}z_0)}\Big)
\le \cfrac{\alpha+1}{2}.
\end{equation}
for some non-zero $z_0\in\mathbb{D}$, since for $z_0=0$ the relation \eqref{eq:14} contradicts to the assumption
$\alpha<1$. Substituting $0\neq z_0=x_0+iy_0\in\mathbb{D}$ in 
(\ref{eq:14}) and simplifying, we obtain
\begin{align*}
2x_0\sqrt{c}\sin (\sqrt{c}x_0) \cos (\sqrt{c}x_0) 
&+2 y_0\sqrt{c} \sinh (\sqrt{c}y_0) \cosh (\sqrt{c}y_0)\\
&\le (1+\alpha)(\sin^2(\sqrt{c}x_0)+\sinh^2(\sqrt{c}y_0)),
\end{align*}
or
\begin{align*}
\sin(\sqrt{c}x_0) &\cos(\sqrt{c}x_0) [2x_0\sqrt{c}-(1+\alpha)\tan(\sqrt{c}x_0)]\\
& \le \sinh(\sqrt{c}y_0) \cosh(\sqrt{c}y_0)[(1+\alpha)\tanh (\sqrt{c}y_0)-2(\sqrt{c}y_0)],
\end{align*}
for $0<c= c_\alpha+\epsilon$ and $x_0^2+y_0^2<1.$  Choose $y_0=0.$ 
Then to obtain our desired inequality, we have to find $x_0 \in (-1,1)$, $x_0\neq 0$, 
such that
\begin{equation}\label{eq:15}
\sin(\sqrt{c}x_0) \cos(\sqrt{c}x_0) [2x_0\sqrt{c}-(1+\alpha)\tan(\sqrt{c}x_0)]\le 0
\end{equation}
holds.
Now, we see that $\sin(\sqrt{c}x_0)$ and $\cos(\sqrt{c}x_0)$ are positive in $(0,\pi/{2\sqrt{c}})$, and $2x_0\sqrt{c}-(1+\alpha)
\tan(\sqrt{c}x_0)$ is negative in $(\sqrt{c_\alpha}/{\sqrt{c}},\pi/2\sqrt{c})$, where
the latter part follows by (\ref{eq:17}). 
Therefore, \eqref{eq:15} holds true for some $x_0$ in the intersection
$$
(0,\pi/{2\sqrt{c}}) \cap (\sqrt{c_\alpha}/{\sqrt{c}},1)\subset (0,1),
$$
since $c_\alpha<c$.
This completes the proof of our first main theorem. 
\hfill{$\Box$}
\end{enumerate}

In the following example, we construct a function meromorphically convex of order
$\alpha$ satisfies the hypothesis of Theorem~\ref{main-thm1}.
\begin{example} For a constant $c>0$,
consider the function $f$ defined by 
$$
f(z)=\frac{w_1(z)}{w_2(z)}=\sqrt{c}\cot (\sqrt{c}z),
$$
where $w_1(z)=\cos(\sqrt{c}z)$ and $w_2(z)=(1/\sqrt{c})\sin(\sqrt{c}z)$ that satisfy the 
differential equation
$$
w''+2cw=0.
$$
By Lemma~$\ref{l:-1}$, it follows that $S_f(z)=2c$. Now, for any such constant $c\le c_\alpha$,
where $c_\alpha$ is the smallest positive root of the equation \eqref{eq:2}, one
clearly sees that
$$
|S_f(z)|\le 2c_\alpha.
$$
Next, by comparing with Lemma~$\ref{l:2}$, we see that
$$
{\rm Re}\Big(\frac{zw_2'(z)}{w_2(z)}\Big)={\rm Re}(z\sqrt{c}\cot(\sqrt{c}z))>\frac{1+\alpha}{2}.
$$
This is equivalent to saying that $f$ is meromorphically convex of order $\alpha$, by 
Lemma~$\ref{l:1}$. Thus, Theorem~$\ref{main-thm1}$ is satisfied by the function $f(z)=\sqrt{c}\cot (\sqrt{c}z)$.
\end{example}

\subsection{Proof of Theorem \ref{main-thm2}}
We adopt the idea from the proof of \cite[Theorem~2]{Chi94}. 
Suppose that $u(z)$ and $v(z)$ are two linearly independent solutions of the differential
equation \eqref{eq:3} with $S_g(z)=2p(z)$, where $u(0)=v'(0)=0$ and
$u'(0)=v(0)=1$. Then by a similar analysis as in the proof of \cite[Theorem~2]{Chi94}, we obtain
$$
g(z)=\cfrac{u(z)}{cu(z)+v(z)},
$$
where $c=-a_2$.
An easy computation yields
\begin{equation}\label{eq:20}
1+\cfrac{zg''(z)}{g'(z)}=1-2z \cfrac{cu'(z)+v'(z)}{cu(z)+v(z)}.
\end{equation}

Now, by the hypothesis, it is easy to see that
$$
\phi(\beta)=\min \left\{\cfrac{\beta-1}{\beta+1}\,,\cfrac{6(\beta-1)}{2(7\beta-9)}\right\}<1 \mbox{ and } \psi(\beta)=\max \left\{\cfrac{\beta+3}{\beta+1}\,,\cfrac{11\beta-15}{7\beta-9}\right\}>1.
$$
Also, we note that 
$$
2\eta+(1+\eta)\delta e^{\delta/2}<2\eta+ \psi(\beta)\delta(1+\eta)e^{\delta/2}<2\phi(\beta)<2
$$
follows from the assumption (\ref{e:4}). Hence $\eta+(1+\eta)\delta e^{\delta/2}/2<1$.
Now \cite[(13)]{Chi94} also satisfied by our hypothesis. Thus, it follows from the similar argument as in
the proof of \cite[Theorem~2]{Chi94} that 
$$
\Big|\cfrac{cu'(z)+v'(z)}{cu(z)+v(z)}\Big| \leq \cfrac{2(\eta+(1+\eta)\delta e^{\delta/2})}{2-2\eta-(1+\eta)\delta e^{\delta/2}},
$$
which yields
\begin{equation}\label{eq:21}
{\rm Re}\Big(\cfrac{z(cu'(z)+v'(z))}{cu(z)+v(z)}\Big)>- \Big|\cfrac{z(cu'+v')}{cu+v}\Big|>  
-\cfrac{2(\eta+(1+\eta)\delta e^{\delta/2})}{2-2\eta-(1+\eta)\delta e^{\delta/2}},
\end{equation}
and 
\begin{equation}\label{eq:24}
{\rm Re}\Big(\cfrac{z(cu'(z)+v'(z))}{cu(z)+v(z)}\Big) \leq \Big|\cfrac{z(cu'+v')}{cu+v}\Big|  <
\cfrac{2(\eta+(1+\eta)\delta e^{\delta/2})}{2-2\eta-(1+\eta)\delta e^{\delta/2}}\,.
\end{equation}
The relations (\ref{eq:20}), (\ref{eq:21}) and (\ref{eq:24}) together lead to

$$
\cfrac{2-6\eta-5(1+\eta)\delta e^{\delta}/2}{2-2\eta-(1+\eta)\delta e^{\delta}/2}<{\rm Re}\Big(1+\cfrac{zg''(z)}{g'(z)}\Big) < \cfrac{2+2\eta+3(1+\eta)\delta e^{\delta}/2}{2-2\eta-(1+\eta)\delta e^{\delta}/2}.
$$
The hypothesis (\ref{e:4}) thus obtains 
$$
\cfrac{2+2\eta+3(1+\eta)\delta e^{\delta}/2}{2-2\eta-(1+\eta)\delta e^{\delta}/2}<\beta
$$
and 
$$
\cfrac{2-6\eta-5(1+\eta)\delta e^{\delta}/2}{2-2\eta-(1+\eta)\delta e^{\delta}/2} > \cfrac{-\beta}{2\beta-3},
$$
completing the proof. \hfill{$\Box$}

\begin{remark}
The constant $\phi(\beta)$ in the statement of Theorem~$\ref{main-thm2}$ is not sharp. For instance,
the function $g(z)=\cfrac{2z-z^2}{2(1-z)^2}\in\mathcal{C}_\infty$ for which $|a_2|=3/2>1$.
\end{remark}

In the following example we construct a function that agree with Theorem~\ref{main-thm2}
for some $\beta\ge 3/2$.
\begin{example}
For any constant $c$ with $|c|<3/7$, consider the function $g$ defined by
$$
g(z)=\frac{z}{1-cz}, \quad |z|<1.
$$
We show that $g\in \mathcal{C}_{5/2}$ and it satisfies the hypothesis of Theorem~\ref{main-thm2}.

First, we note that $g$ is a M\"obius transformation and hence $S_g=0$. Therefore, it
trivially satisfies the hypothesis of Theorem~\ref{main-thm2}.

Secondly, an easy computation yields
$$
1+\frac{zg''(z)}{g'(z)}=\frac{1+cz}{1-cz}.
$$
From this, we have
$$
{\rm Re}\Big(1+\frac{zg''(z)}{g'(z)}\Big)=\frac{1-|c|^2|z|^2}{|1-cz|^2}.
$$
By the usual triangle inequalities, it follows that
$$
\frac{1-|c||z|}{1+|c||z|}\le \frac{1-|c|^2|z|^2}{|1-cz|^2}\le \frac{1+|c||z|}{1-|c||z|}.
$$
Since $|c|<3/7$, for $|z|<1$, it is easy to verify that
$$
\frac{1+|c||z|}{1-|c||z|}<\frac{5}{2}
~~\mbox{ and }~~
-\frac{5}{4}<\frac{1-|c||z|}{1+|c||z|}
$$
hold true. Thus, $g\in \mathcal{C}_{5/2}$.
\end{example}

\noindent
{\bf Acknowledgement.} The second author acknowledges some useful suggestions made by Prof. S. Ponnusamy on this manuscript and for brining the paper \cite{SHAH73} to his attention.

\end{document}